\documentclass[11pt]{amsart}
\usepackage{amssymb,latexsym,amsmath,amsfonts,amsthm,amscd,graphicx,url,color,hyperref,stmaryrd}

\theoremstyle{plain}

\newtheorem{thm}{Theorem}
\newtheorem{prop}[thm]{Proposition}
\newtheorem{cor}[thm]{Corollary}
\newtheorem{lemma}[thm]{Lemma}
\newtheorem{que}[thm]{Question}

\theoremstyle{definition}

\newtheorem{remark}[thm]{Remark}

\newtheorem{example}[thm]{Example}

\newcommand{\mm}{\mathfrak{m}}
\newcommand{\G}{\mathrm{gr}_\mm(R)}
\newcommand{\LM}{\mathrm{LM}}

\newcommand{\Card}{\mathrm{Card}}

\newcommand{\codim}{\mathrm{codim}}
\newcommand{\Tor}{\mathrm{Tor}}

\title{On the multiplicity of tangent cones of monomial curves}
\author{Alessio Sammartano}
\address{Department of Mathematics, University of Notre Dame, 255 Hurley, Notre Dame, IN 46556, USA}
\email{asammart@nd.edu}


\begin{document}

\begin{abstract}
Let $\Lambda$ be a numerical semigroup, $\mathcal{C}\subseteq \mathbb{A}^n$  the monomial curve singularity associated to $\Lambda$, and $\mathcal{T}$ its tangent cone.
In this paper we provide a sharp upper bound for the least positive integer in $\Lambda$ in terms of the codimension and the maximum degree of the equations of $\mathcal{T}$,
when $\mathcal{T}$ is not a complete intersection. 
A special case of this result settles a question of J. Herzog and D. Stamate.
\end{abstract}

\maketitle

\section{Introduction}

Let $G$ be a standard graded algebra over a field $\Bbbk$.
It is an important problem in commutative algebra and algebraic geometry
to find formulas and inequalities that relate the {multiplicity} or {degree} $e(G)$ to other invariants of $G$,
such as the codimension, degrees of the defining equations, or degrees of the  higher syzygies.
Significant advancements have been achieved in this area in recent years, 
see for instance \cite{BS12,CCMPV17,ES09,MP17}.

There is an obvious upper bound for $e(G)$ when $G$ has codimension $c$ and is defined by forms of degree $d$, namely $e(G) \leq d^c$, 
with equality holding if and only if $G$ is a complete intersection.
We will assume that $G$ is not a complete intersection,  then the  question becomes how close can  $e(G)$ actually be to $d^c$.
A general result in this direction was proved in \cite{E09} for almost complete intersections, see also \cite{HMMS15}.

It is interesting to investigate this problem in special situations.
In this paper, we are concerned with the case when $G$ is  the tangent cone  of  an affine monomial curve singularity.
Let $\Lambda = \langle n_0, \ldots, n_c\rangle$ be a numerical semigroup,
i.e. a cofinite submonoid of $(\mathbb{N},+)$.
The  ring $R= \Bbbk \llbracket \Lambda \rrbracket = \Bbbk \llbracket t^n \, : \, n \in \Lambda \rrbracket \subseteq \Bbbk \llbracket t\rrbracket$ 
is the completed local ring at the origin of the  curve $\mathcal{C}\subseteq\mathbb{A}^{c+1}$ parametrized by $X_0 = t^{n_0}, \ldots, X_c = t^{n_c}$.
The  associated graded ring of $R$ is $\G = \oplus_{i\geq 0}\mm^i/\mm^{i+1}$, where $\mm$ denotes the maximal ideal of $R$;
it is the coordinate ring of the tangent cone $\mathcal{T}$ to $\mathcal{C}$ at the singularity $\mathbf{0} \in \mathcal{C}$.
The combinatorial structure of such rings allows to formulate more precise results,
and thus the study of the algebraic properties of $\G$ is an active topic of research, 
with  recent progress concerning especially its  defining equations \cite{DMS13,HS14,HS17,K16},
Hilbert function \cite{AMS09,OST17},
 Cohen--Macaulayness \cite{BF06,H16,JZ16,S16},
and   related notions \cite{Br10,CZ09,DMS11}.
Notice that if $\Lambda$ is minimally generated by $n_0 <  \cdots < n_c$
then we have $\codim(R) = \codim(\G) = c$ and $e(R) = e(\G ) = n_0$.
In the case  $c=1$ of plane curves, $\G$ is always a complete intersection, so we assume $c \geq 2$.

In a recent work  \cite{HS17} Herzog and Stamate consider  tangent cones  defined by quadrics, calling ``quadratic'' those numerical semigroups for which this condition occurs.
They propose the following problem:

\begin{que}[{\cite[Question 1.11]{HS17}}]\label{Question}
Let $\Lambda = \langle n_0, \ldots, n_c \rangle$ be a numerical semigroup, $R = \Bbbk \llbracket \Lambda\rrbracket$ and $G = \G$.
Assume that $G$ is defined by quadratic equations. 
Are the following statements true?
\begin{itemize}
\item[(a)]
Either $ e(G) \leq 2^c- 2^{c-2}$ or $e(G) = 2^c$. 
\item[(b)]
If $e(G) = 2^c- 2^{c-2}$, then $G$ is an almost complete intersection.
\end{itemize}
\end{que}
\noindent
This question  was motivated in part
 by experimental evidence and by an affirmative answer  under the assumption that $G$ is Cohen--Macaulay \cite[Theorem 1.9]{HS17}.
Observe that the proposed inequality  is sharper than the one of \cite[Theorem 1]{E09}, which in this case would yield  $ e(G) \leq 2^c- c+1$.

 In this paper, we give a complete affirmative answer  to Question \ref{Question}.
In fact, Theorem \ref{Theorem} provides  a bound on the multiplicity of non-complete intersection tangent cones whose degrees of the equations are bounded by an arbitrary  $d\geq 2$;
the case $d=2$ corresponds to Question \ref{Question}.
Furthermore, we prove  a stronger statement than item (b).
In Proposition \ref{PropositionExtremal} we  determine the minimal free resolutions of $R$ and $G$ in the case when the bound is achieved.
Finally, we show in Proposition \ref{Proposition} that the bound is sharp.

\section{Main result}

This section is devoted to the proof of the main theorem. 
We refer to \cite{E95} for definitions and background.

We will need the 
 following  discrete optimization results:

\begin{lemma}\label{Lemma}
Let $c, d\in \mathbb{N} $ be such that $c, d \geq 2$ and $E = \{\varepsilon_1, \ldots, \varepsilon_c\} \subseteq \mathbb{N}$  a multiset with $1 \leq \varepsilon_i \leq d$ for every $i$.

\begin{enumerate}
\item If $\sum_{i=1}^c\varepsilon_i = (c-1)d$ then $\prod_{i=1}^{c} \varepsilon_i \geq (d-1)d^{c-2}$,
with equality if and only if $E=\{1,d-1, d, \ldots, d\}$.

\item If $\sum_{i=1}^c\varepsilon_i = (c-1)d+1$ then $\prod_{i=1}^{c} \varepsilon_i\geq d^{c-1}$.

\end{enumerate}
\end{lemma}
\begin{proof}
The two proofs are  similar, so we only present the first one.
Suppose without loss of generality that $ \varepsilon_1 \leq \cdots \leq \varepsilon_c$.
If $c=d=2$ then the result is obvious, so we assume $c+d \geq 5$.
Define
$$
\rho = \min \{ i \, : \, \varepsilon_i > 1\}, \qquad \tau = \max \{ i \, : \, \varepsilon_i < d\}.
$$
Since $ c <\sum_{i=1}^c\varepsilon_i <cd$,  both indices are well-defined integers in $\{1, \ldots, c\}$. 

If $\rho > \tau $ then $\rho = \tau +1 $ and  $\sum_{i=1}^c\varepsilon_i  = \tau + (c-\tau)d$;
however, since $\sum_{i=1}^c\varepsilon_i  =  (c-1)d$, we derive $(\tau-1) d = \tau$ and thus $d=\tau = 2$.
In this case,  $E = \{1,1,2,\ldots,2\}$ and the equality $\prod_{i=1}^{c} \varepsilon_i = (d-1)d^{c-2}$ holds.

If $\rho = \tau$ then we have $\sum_{i=1}^c\varepsilon_i = (\tau-1) + \varepsilon_\tau + (c-\tau)d$, which forces 
$ \varepsilon_\tau = (\tau-1) (d-1)$.
However, we have $2 \leq \varepsilon_\tau \leq d-1$, whence $\tau = 2$ and  $\varepsilon_2 = d-1$ .
Thus, in this case  $E = \{1,d-1,d,\ldots,d\}$ and the equality $\prod_{i=1}^{c} \varepsilon_i =  (d-1)d^{c-2}$ holds.

If $\rho < \tau$ then we define another multiset of integers $\varepsilon'_i$ by setting $\varepsilon'_\rho = \varepsilon_\rho-1,\varepsilon'_\tau = \varepsilon_\tau+1, \varepsilon'_i = \varepsilon_i$ for $i \ne \rho, \tau$.
Since $\varepsilon_\rho + \varepsilon_\tau =  \varepsilon'_\rho + \varepsilon'_\tau $ and  $0 \leq \varepsilon_\tau - \varepsilon_\rho <  \varepsilon'_\tau - \varepsilon'_\rho $ we have 
$\varepsilon_\rho \cdot \varepsilon_\tau > \varepsilon'_\rho \cdot \varepsilon'_\tau$ and hence
$\prod_{i=1}^{c} \varepsilon_i  > \prod_{i=1}^{c} \varepsilon'_i$. 
However, we still have $\sum_{i=1}^c\varepsilon'_i  =  (c-1)d$,
thus, by double induction on $\tau-\rho $ and $\varepsilon_\tau - \varepsilon_\rho$, we conclude that $ \prod_{i=1}^{c} \varepsilon'_i \geq (d-1)d^{c-2}$.
The desired statements follow.
 \end{proof}

We are ready to present the main result of this paper.

\begin{thm}\label{Theorem}
Let $\Lambda\subseteq \mathbb{N}$ be a numerical semigroup, $R= \Bbbk\llbracket \Lambda \rrbracket $, and  $G = \G$  the associated graded ring.
Assume that $G$ has codimension $c$ and is defined by equations of degree at most $d$, for some  $c,d\geq 2$.
If $G$ is not a complete intersection, then the multiplicity of $G$ satisfies
$$
e(G) \leq d^c - (d-1)d^{c-2}.
$$
Furthermore, if equality holds then $G$ is a Cohen--Macaulay almost complete intersection.
\end{thm}

\begin{proof}
Let $n_0 < \cdots < n_c$ be the minimal set of generators of $\Lambda$, and define a regular presentation 
$\pi_R :  \Bbbk\llbracket X_0, \ldots, X_c \rrbracket \twoheadrightarrow R= \Bbbk\llbracket \Lambda\rrbracket$
by $\pi_R(X_i) = t^{n_i}$.
This induces a regular presentation $ \pi_G : P = \Bbbk[ x_0, \ldots, x_c] \twoheadrightarrow G$.
Let $I = \ker(\pi_R)$ and $J= \ker(\pi_G)$ denote the defining ideal of $R$ and $G$, respectively.
The ideal $I$ is generated by all binomials $\prod_{j=0}^c X_j^{\alpha_j}-\prod_{j=0}^c X_j^{\beta_j}$ with $\sum_{j=0}^c \alpha_j n_j = \sum_{j=0}^c \beta_j n_j$.
The ideal $J$ is generated by the initial forms of  elements of $I$,
so by binomials and monomials.
By assumption, the minimal generators of $J$ have degree at most $d$.

For each $i = 1, \ldots, c$, define
$$
d_i = \inf \Big\{ \delta \in \mathbb{N}^+ \, : \, 0 \ne X_i^{\delta} - \prod_{j=0}^c X_j^{\alpha_j} \in I \text{ for some } \alpha_j \in \mathbb{N} \text{ with } \sum_{j=0}^c \alpha_j \geq \delta \Big\}.
$$
First of all, we observe that $d_i < \infty$, since $X_i^{n_0} - X_0^{n_i} \in I$.
By definition, $d_i$ is the lowest degree of a form $g \in J$ containing a pure power of $X_i$ in its support.
It follows in particular that $d_i \leq d$, otherwise $J$ would have a minimal generator of degree $d_i > d$, 
giving a contradiction.
Up to multiplying by $X_i^{d-d_i}$, we conclude that there exists a nonzero binomial $f_i = X_i^d -  \prod_{j=0}^c X_j^{\alpha_j} \in I$ such that $ \sum_{j=0}^c \alpha_j \geq d$.
Let $g_i \in J$ be the initial form of $f_i$.
Note that either  $g_i = x_i^d -  \prod_{j=0}^c x_j^{\alpha_j} $ or $g_i = x_i^d$, depending on whether $ \sum_{j=0}^c \alpha_j = d$ or $ \sum_{j=0}^c \alpha_j > d$.

Let $\prec$ denote the reverse lexicographic monomial order on $P$ with the variables ordered by  $x_c \succ x_{c-1} \succ \cdots \succ x_0$,
and denote leading monomials by $\LM_\prec(-)$.
Then we have  $\LM_\prec(g_i) = x_i^d$ for every $i=1, \ldots, c$:
this is obvious for those $i$ such that $f_i$ is not homogeneous.
If $f_i = X_i^d -  \prod_{j=0}^c X_j^{\alpha_j} \in I$ with  $ \sum_{j=0}^c \alpha_j = d$,
then necessarily $\alpha_j >0 $ for some $j < i$, as the generators of $\Lambda$ are in increasing order; 
we conclude that $x_i^d \succ \prod_{j=0}^c x_j^{\alpha_j} $.

Since the sub-ideal $(x_1^d, \ldots, x_c^d) \subseteq \LM_\prec(J)$ is generated by a regular sequence, 
 by upper semicontinuity  the sub-ideal $J'=(g_1, \ldots, g_c)\subseteq J$ is also generated by a regular sequence.
By assumption $J$ is generated in degrees at most $d$ and it cannot be generated by a regular sequence, 
therefore $J' $ and $J$ must differ in some degree $d'\leq d$.
However, this implies that $J'$ and $J$ differ in degree $d$.
In fact, the quotient $P/J'$ is Cohen-Macaulay of dimension 1, so the local cohomology $H^0_{\mm_P}(P/J')$  vanishes, 
and the non-zero submodule $J/J'\subseteq P/J'$ contains no non-trivial submodule of finite length. 
We conclude that  there exists a homogeneous $g_0 \in J \setminus (g_1, \ldots, g_c)$, and it may be chosen to be a monic monomial or binomial of degree $d$.

Suppose that $\LM_\prec(g_0) = x_i^d$ for some $i = 1, \ldots, c$. 
Then $g_0 - g_i\ne 0$ has degree $d$ and it does not contain any pure power in its support,  as no homogeneous binomial in $I$ contains two distinct pure powers in its support.
Thus, up to replacing $g_0$ with $g_0-g_i$, we may assume that $\LM_\prec(g_0)=\mathcal{M}$ is a monomial of degree $d$ divisible by at least two distinct variables of $P$.

We obtain the inclusion of monomial ideals 
$$
H = \big(\LM_\prec(g_0), \LM_\prec(g_1), \ldots, \LM_\prec(g_c)\big) \subseteq L =  \big(\LM_\prec(g) \, : \, g \in J \big).
$$
Now we distinguish two cases.

\underline{Case 1}:
$x_0$ does not divide $\mathcal{M}$.
Then $H$ is a $(x_1, \ldots, x_c)$-primary ideal of $P$ of dimension 1. The variable $x_0$ is a non-zerodivisor on $P/H$,  hence the multiplicity can be computed as 
\begin{eqnarray*}
e\left(\frac{P}{H}\right) &=&e\left(\frac{P}{H+(x_0)}\right) = \dim_\Bbbk\left(\frac{P}{H+(x_0)}\right) = \dim_\Bbbk\left(\frac{\Bbbk[x_1, \ldots,x_c]}{(\mathcal{M},x_1^d, \ldots, x_c^d)}\right)  \\
& = &  \dim_\Bbbk\left(\frac{\Bbbk[x_1, \ldots,x_c]}{(x_1^d, \ldots, x_c^d)}\right)  -  \dim_\Bbbk\left(\frac{(\mathcal{M},x_1^d, \ldots, x_c^d)}{(x_1^d, \ldots, x_c^d)}\right).  
\end{eqnarray*}
The first quantity in this difference is equal to $d^c$, whereas,
writing $\mathcal{M}= \prod_{j=1}^c x_j^{\mu_j}$ where $\sum_{j=1}^c \mu_j = d $ and $\mu_j<d$ for all $j$,
the second quantity is 
$$
\dim_\Bbbk\left(\frac{(\mathcal{M},x_1^d, \ldots, x_c^d)}{(x_1^d, \ldots, x_c^d)}\right) = \Card \big\{ (\gamma_1, \ldots, \gamma_c) \, : \, \mu_j \leq \gamma_j < d\big\}
= \prod_{j=1}^c(d-\mu_j).
$$
By Lemma \ref{Lemma}, the least possible value of the product $\prod_{j=1}^c(d-\mu_j)$ with $1 \leq d -\mu_j \leq d$ and subject to the constraint $\sum_{j=1}^c(d-\mu_j)= cd - \deg(\mathcal{M}) = (c-1)d$ is $(d-1)d^{c-2}$.
In conclusion, we have
$e(P/H) \leq d^c - (d-1)d^{c-2}$.

\underline{Case 2}:
$x_0$ divides $\mathcal{M}$.
Write $\mathcal{M}= x_0^\alpha \mathcal{N}$ where $\mathcal{N}$ is a monomial in $x_1, \ldots, x_c$ of degree $d-\alpha$, with $0<\alpha<d$.
A shortest primary decomposition of $H$ is then $H = H_1 \cap H_2$,
where $H_1 = (x_0^\alpha, x_1^d, \ldots, x_c^d)$ and $H_2 = (\mathcal{N}, x_1^d, \ldots, x_c^d)$.
Note that $H_1$ has dimension 0 while $H_2$ has dimension 1.
It follows, for instance from the associativity formula of multiplicity \cite[Ex. 12.11.e]{E95}, that $e(P/H) = e(P/H_2)$.
The variable $x_0$ is a non-zerodivisor on $P/H_2$, and as above we compute
\begin{eqnarray*}
e\left(\frac{P}{H}\right) &=& e\left(\frac{P}{H_2}\right) = e\left(\frac{P}{H_2+(x_0)}\right) = \dim_\Bbbk\left(\frac{\Bbbk[x_1, \ldots,x_c]}{(\mathcal{N},x_1^d, \ldots, x_c^d)}\right)  \\
& = &  \dim_\Bbbk\left(\frac{\Bbbk[x_1, \ldots,x_c]}{(x_1^d, \ldots, x_c^d)}\right)  -  \dim_\Bbbk\left(\frac{(\mathcal{N},x_1^d, \ldots, x_c^d)}{(x_1^d, \ldots, x_c^d)}\right) 
\end{eqnarray*}
Proceeding as in Case 1, we estimate the second quantity in this difference to be at least $d^{c-1}$, since now we have $\deg(\mathcal{N}) \leq d-1$.
Therefore, in this case we have $e(P/H) \leq d^c - d^{c-1}$.

In either case we see that $e(P/H) \leq d^c - (d-1)d^{c-2}$.
Since both ideals $H$ and $L$ have codimension $c$,
the inclusion $H \subseteq L$ implies that $e(P/L) \leq e(P/H)$.
Finally, the fact that a homogeneous ideal and its initial ideal have the same multiplicity yields
 $e(G) = e(P/J) = e(P/L) \leq e(P/H) \leq d^c - (d-1)d^{c-2}$ as desired.

Now suppose that the equality $e(G)= d^c - (d-1)d^{c-2}$ holds, then necessarily
 $ e(P/L) = e(P/H) = d^c - (d-1)d^{c-2}$.
In particular, Case 2 cannot occur, hence
$P/H$ is  Cohen--Macaulay since $x_0$ is a non-zerodivisor by Case 1.
We have an inclusion of ideals  $H\subseteq L$ of the same codimension $c$ and with the smaller one being $(x_1, \ldots, x_c)$--primary;
the associativity formula of multiplicity forces $H=L$.
We deduce that the initial ideal of $J$ is a Cohen--Macaulay almost complete intersection;
by upper semicontinuity, the same must be true for $J$ itself.
This concludes the proof.

We also observe that the equality $e(P/H) = d^c - (d-1)d^{c-2}$ forces 
the product $\prod_{j=1}^c(d-\mu_j)$ to achieve the least possible value $(d-1)d^{c-2}$.
By Lemma \ref{Lemma} (1), up to renaming the variables $x_1, \ldots, x_c$, we necessarily have $\mathcal{M}=x_1^{d-1}x_2$.
Thus, 
if  the equality $e(G)= d^c - (d-1)d^{c-2}$ holds,
then we have $L = (x_1^d, x_2^d, \ldots, x_c^d, x_1^{d-1}x_2)$, up to renaming the variables.
\end{proof}

Theorem \ref{Theorem} can be applied readily in the negative direction.

\begin{example}
Let $\Lambda = \langle 100, n_1, n_2, n_3 , n_4\rangle$ be minimally generated by  $100< n_1 < n_2 < n_3 < n_4$.
From Theorem \ref{Theorem}  we  deduce that $\G$ must have a minimal  relation of degree at least $ 4$.
\end{example}

We remark that the converse of the last statement in Theorem \ref{Theorem} is false. 
That is,
if $\G$ is a Cohen--Macaulay almost complete intersection, it may happen that $e(\G) < d^c - (d-1)d^{c-2}$.
This is the case for instance for the quadratic semigroup $\Lambda= \langle 11, 13, 14, 15, 19 \rangle$, see also \cite[Remark 1.10]{HS17}.

\section{The extremal case}

In this section we investigate further the case when the upper bound in Theorem \ref{Theorem} is attained, showing that this condition forces very strong properties.

First, we determine the minimal free resolutions of  the semigroup ring $R$ and  tangent cone $G$.

\begin{prop}\label{PropositionExtremal}
Let $\Lambda\subseteq \mathbb{N}$ be a numerical semigroup, $R= \Bbbk\llbracket \Lambda \rrbracket $, and  $G = \G$.
Assume that $G$ has codimension $c\geq 2$, is defined by equations of degree at most $d\geq 2$, and is not a complete intersection.
If  $ e(G) = d^c - (d-1)d^{c-2}$ then the Betti numbers of $R$ and $G$ are
$$
\beta_i(R) = \beta_i(G) = \binom{c-2}{i}+3\binom{c-2}{i-1}+2\binom{c-2}{i-2}\qquad \mbox{ for }\,  i = 0, \ldots, c.
$$
\end{prop}
\begin{proof}
We have shown at the end of the proof of Theorem \ref{Theorem} that, under these assumptions and up to renaming the variables, 
we have $L=\LM_\prec(J) = (x_1^d, x_2^d, \ldots, x_c^d, x_1^{d-1}x_2)$. 
We determine the minimal graded free resolution of $P/L$.
Observe that $P/L \cong P'/L' \otimes_\Bbbk P''/L''$ where $L' = ( x_1^{d}, x_1^{d-1}x_2, x_2^d) \subseteq P' = \Bbbk[x_0,x_1, x_2]$ and 
$L'' = ( x_3^{d}, x_4^{d}, \ldots, x_c^d) \subseteq P'' = \Bbbk[x_3,\ldots, x_c]$.
The ideal $L'$ is perfect of codimension 2, hence the resolution of $ P'/L' $ over $P'$ is determined by the Hilbert-Burch matrix
$$
\begin{pmatrix}
x_2 & 0 \\
-x_1 & x_2^{d-1} \\
0 & -x_1^{d-1}
\end{pmatrix}
$$
and therefore it has the form
$$
0 \rightarrow P'(-2d+1)\oplus P'(-d-1) \rightarrow P'(-d)^3 \rightarrow P'.
$$
The ideal $L''$ is generated by a regular sequence and  the resolution of $ P''/L'' $ over $P''$ is given by the Koszul complex
$$
0 \rightarrow P''(-(c-2)d) \rightarrow \cdots  \rightarrow P''(-2d)^{\binom{c-2}{2}} \rightarrow P''(-d)^{c-2} \rightarrow P''.
$$
Finally, the minimal free resolution of $P/L$ over $P$ is obtained by tensoring the two resolutions, 
hence the graded Betti numbers are given by
$$
\beta_{i,j}(P/L) = \sum_{\substack{i'+i'' = i \\ j' + j'' = j}}\beta_{i',j'}(P'/L')\cdot\beta_{i'',j''}(P''/L'').
$$
We obtain the following formulas for  the nonzero graded Betti numbers of $P/L$: 
if $d \geq 3$ then
\begin{eqnarray*}
\beta_{i,id}(P/L) &=& \binom{c-2}{i}+3\binom{c-2}{i-1}\\
\beta_{i,id-1}(P/L) &=& \binom{c-2}{i-2}\\
 \beta_{i,(i-1)d+1}(P/L) &=& \binom{c-2}{i-2}
\end{eqnarray*}
whereas if $d = 2$  one simply adds the last two lines. 

The formulas above imply that $\beta_{i,j}(P/L)\cdot \beta_{i+1,k}(P/L) = 0 $ for all $k \leq j$.
In other words, we cannot have any consecutive cancellation of the same degree \cite{Pe04} or of negative degree \cite{RS10,Sa16}.
It follows from \cite[Proof of Theorem 1.1]{Pe04} that 
$\beta_{i,j}(P/L) = \beta_{i,j}(P/J)$ for all $i,j$.
As for the  Betti numbers of the local ring $R$, 
it follows from \cite[Theorem 3.1]{RS10} or \cite[Theorem 2]{Sa16} that 
$\beta_{i}(G) = \beta_{i}(R)$ for all $i$.
The proof is concluded.
\end{proof}

\begin{remark}
We have proved that, if the upper bound is attained, then $G$, and therefore $R$, are almost complete intersections.
Moreover, their defining ideals $J$ and $I$ are \emph{licci},
i.e. they belong to the linkage class of a complete intersection.
In fact,  as already observed,  $L = \LM_\prec(J) = (x_1^d, \ldots, x_c^d, x_1^{d-1}x_2)$ and we have
 $ (x_1^d, \ldots, x_c^d) : L =  (x_1^d, \ldots, x_c^d) : ( x_1^{d-1}x_2) = (x_1, x_2^{d-1}, x_3^d, \ldots, x_c^d)$,
that is, $L$ is linked in one step to a complete intersection;
this implies that the same is true for  $J$ and $I$.
Furthermore, $J$ is \emph{strongly licci} in the sense of \cite{HU07}.
\end{remark}

Next,
we construct a family of monomial curves to show that the upper bound for the multiplicity is sharp.

\begin{prop}\label{Proposition}
For every $c,d\geq 2$  there exists a numerical semigroup attaining the upper bound in Theorem \ref{Theorem}.
\end{prop}

\begin{proof}
Let $ e = d^c - (d-1)d^{c-2}$ and set
$$
n_0 = e,\quad n_1 = e+1, \quad n_2 = e + d,\quad n_i = e + (d^2-d+1)d^{i-3} \quad \mbox{for }3 \leq i \leq c.
$$
Consider the numerical semigroup $\Lambda = \langle n_0, \ldots, n_c\rangle$, and notice that the generating set is minimal because $n_0 < \cdots < n_c < 2n_0$.
Clearly, we have $e(G) = d^c - (d-1)d^{c-2}$ and $\codim(G)=c$; it remains to show that $G$ is defined by relations of degrees at most equal to $d$.

We use the same notation as in the  proof of Theorem 	\ref{Theorem}.
If $c\geq 3$ then the defining ideal $I$ of $R$ contains the relations
$f_0 = X_1X_2^{d-1} - X_0^{d-1} X_3$, 
$f_2 = X_2^d - X_1^{d-1}X_3$,
$f_c = X_c^d - X_0^{d+1}$,
and 
$f_i =X_i^d - X_{i-1}^{d-1}X_{i+1}$ for $i=1 ,3, 4, \ldots, c-1$.
If $c=2$ then $I$ contains  $f_0 = X_1X_2^{d-1} - X_0^{d+1}, f_1 = X_1^{d} - X_0^{d-1}X_2, f_2 = X_2^{d} - X_0^2X_1^{d-1}$.
Let $g_i$ be the initial form of $f_i$ and let $H = \big( \LM_\prec(g_0), \ldots, \LM_\prec(g_c)\big) = \big( x_1x_2^{d-1}, x_1^d, \ldots, x_c^d\big)$.
As in the proof of  Theorem 	\ref{Theorem} we see that
 $H$ is a primary ideal of codimension $c$ and multiplicity $d^c - (d-1)d^{c-2}$, 
and then $H$ must coincide with the initial ideal of the defining ideal $J$ of $G$.
In particular, $J= (g_0, \ldots, g_c)$.
\end{proof}

A standard graded $\Bbbk$--algebra $G$ is called Koszul if $\Bbbk$ has a linear $G$--resolution, equivalently if $\Tor_i^G(\Bbbk, \Bbbk)_j = 0$ for all $i\ne j$,
cf. \cite{CDR13}.
If $G$ is a Koszul algebra then it is  defined by quadrics, however, this is only a necessary condition.
It is interesting to find sufficient conditions for quadratic $\Bbbk$--algebras to be Koszul.
The next corollary shows that attaining the upper bound in Theorem \ref{Theorem} is a sufficient condition.

\begin{cor}
Let $\Lambda $ be a quadratic numerical semigroup  minimally generated by $n_0 < \ldots < n_c $. 
If $n_0 = 2^c - 2^{c-2} $, then $\G$ is a Koszul algebra.
\end{cor}

\begin{proof}
From the proof of the last statement of Theorem \ref{Theorem} with $d=2$, 
we see that the defining ideal $J$ of $\G$ has an initial ideal $L$ generated by quadratic monomials;
this implies the Koszul property, cf. \cite{CDR13}.
\end{proof}

We conclude the paper with a general discussion.

\begin{remark}\label{Remark}
It is natural to ask for what classes of  rings Theorem \ref{Theorem} is valid.
No example of  a standard graded $\Bbbk$--algebra $G$ is known which violates the upper bound.
In fact, it is possible to show that if the  Eisenbud--Green--Harris conjecture \cite{EGH93} holds, then the inequality is true for any standard graded $\Bbbk$--algebra $G$.
We refer to \cite{FR07} for a detailed account of this problem.
Roughly speaking,
the most general formulation predicts that every Hilbert function in a complete intersection defined by forms of prescribed degrees
is realized by a lexsegment ideal in a complete intersection defined by pure powers  of the given degrees.
The conjecture has been solved only in some special cases, e.g. \cite{A15,CCV14,CM08,C16,O02}.
\end{remark}

\subsection*{Acknowledgments}
Part of this work was supported by the National Science Foundation under Grant No. 1440140, while the author was a Postdoctoral Fellow at the Mathematical Sciences Research Institute in Berkeley, CA.
The author would like to thank Giulio Caviglia, Dumitru Stamate, Francesco Strazzanti,  and especially an anonymous referee for some helpful comments.

\end{document}